\documentclass[12pt]{amsart}
\usepackage{amssymb}
\usepackage{amsmath}
\usepackage{longtable}
\newcommand\g{{\mathfrak g}}
\newcommand\h{{\mathfrak h}}

\newcommand\codim{\operatorname{codim}}

\newcommand\im{\operatorname{im}}
\newcommand\Spec{\operatorname{Spec}}
\newcommand\Quot{\operatorname{Quot}}

\newcommand\K{\mathbb K}

\newcommand\Q{\mathbb Q}

\newcommand\N{\mathbb N}
\newcommand\Z{\mathbb Z}
\newcommand\D{\mathcal D}

\newcommand\Sp{\mathop{\rm Sp}\nolimits}

\renewcommand{\mod}{\mathop{\rm mod}\nolimits}

\newcommand{\ZP}{\operatorname{Z}}
\newcommand{\ZA}{\mathcal{Z}}
\newcommand\quo{/\!/}
\newtheorem{Thm}{Theorem}[section]
\newtheorem{Prop}[Thm]{Proposition}

\newtheorem{Lem}[Thm]{Lemma}
\theoremstyle{definition}
\newtheorem{Ex}[Thm]{Example}

\newtheorem{defi}[Thm]{Definition}

\unitlength=1mm   \numberwithin{equation}{section}
\numberwithin{table}{section} \oddsidemargin=0cm
\author{Ivan Losev}
\title{Lifting central invariants of quantized Hamiltonian actions}
 \thanks{Partially supported by A. Moebius
foundation}
\begin{document}
\begin{abstract}
Let $G$ be a connected reductive group over an algebraically closed
field $\K$ of characteristic $0$, $X$ an affine symplectic variety
equipped with a Hamiltonian action of $G$. Further, let $*$ be a
$G$-invariant Fedosov star-product on $X$ such that the Hamiltonian
action is quantized. We establish an isomorphism between  the center
of the quantum algebra $\K[X][[\hbar]]^G$ and the algebra of formal
power series with coefficients in the Poisson center of $\K[X]^G$.
\end{abstract}
\maketitle
\section{Introduction}\label{SECTION_intro}
In this paper we establish a relation between the centers of certain
Poisson algebras and their quantizations.  Poisson algebras in
interest are invariant algebras for Hamiltonian actions of reductive
algebraic groups on  affine symplectic varieties (the definition of
a Hamiltonian action will be given in the beginning of Section
\ref{SECTION_Hamil}). All varieties and groups are defined over an
algebraically closed field $\K$ of characteristic $0$.

Until further notices $G$ is a connected reductive algebraic group,
$X$ is an affine symplectic variety equipped with  a Hamiltonian
action of $G$. Construct a Fedosov star-product on $X$ and suppose
that the Hamiltonian action can be quantized (all necessary
definitions concerning star-products and quantized Hamiltonian
actions are given in Section \ref{SECTION_quant}). Let
$U^\wedge_\hbar(\g)$ stand for the completed homogeneous universal
enveloping algebra of $\g$ (see the end of Section 3). We have a
natural algebra homomorphism $U^\wedge_\hbar(\g)\rightarrow
\K[X][[\hbar]]$ mapping $\xi\in\g$ to $\widehat{H}_\xi$.

Let $\ZP(\bullet),\ZA(\bullet)$ stand for the center of  Poisson and
associative algebras, respectively. In a word, the main result of
this paper is that the algebras
$\ZP(\K[X]^G)[[\hbar]],\ZA(\K[X][[\hbar]]^G)$ are isomorphic. In
fact, a more precise statement holds. Let us state it.

 We have (see Section
\ref{SECTION_quant}) the following commutative diagram, where the
vertical arrow is an isomorphism of topological
$\K[[\hbar]]$-algebras.

\begin{equation}\label{eq:1}
\begin{picture}(60,30)
\put(2,22){$S(\g)^\g[[\hbar]]$} \put(2,2){$\ZA(U^\wedge_\hbar(\g))$}
\put(35,13){$S(\g)^\g$}
\put(8,20){\vector(0,-1){13}}\put(10,13){{\tiny $\cong$}}
\put(20,23.5){\vector(3,-2){13}} \put(20,4){\vector(3,2){13}}
\end{picture}
\end{equation}

 The main result of this paper is the
following theorem.

\begin{Thm}\label{Thm:1.1}
Identify $S(\g)^\g[[\hbar]],\ZA(U^\wedge_\hbar(\g))$ by means of any
topological $\K[[\hbar]]$-algebra isomorphism making the previous
diagram commutative.  Then there exists an isomorphism
$\ZP(\K[X]^G)[[\hbar]]\rightarrow \ZA(\K[X][[\hbar]]^G)$ of
$\K[[\hbar]]$-algebras such that the following diagram is
commutative.

\begin{picture}(100,30)
\put(2,22){$S(\g)^\g[[\hbar]]$} \put(2,2){$\ZA(U^\wedge_\hbar(\g))$}
\put(35,22){$\ZP(\K[X]^G)[[\hbar]]$}\put(36,2){$\ZA(\K[X][[\hbar]]^G)$}\put(73,13){$\ZP(\K[X]^G)$}
\put(50,20){\vector(0,-1){13}}
\put(61,23){\vector(3,-2){11}}\put(61,5){\vector(3,2){11}}
\put(8,20){\vector(0,-1){13}}\put(10,13){{\tiny $\cong$}}
\put(20,23.5){\vector(1,0){14}} \put(20,4){\vector(1,0){14}}
\end{picture}
\end{Thm}

Let us describe some results  related to ours.

In \cite{Kontsevich} Kontsevich quantized the algebra of smooth
functions on an arbitrary real Poisson manifold, say $M$. His
construction yields an isomorphism
$\ZP(C^\infty(M))[[\hbar]]\rightarrow \ZA(C^\infty(M)[[\hbar]])$.

Also let us mention  the Duflo conjecture, \cite{Duflo}. Let $G$ be
a Lie group, $H$ its closed subgroup. Let $\g,\h$ denote the Lie
algebras of $G,H$, respectively.  Further, choose $\lambda\in
(\h/[\h,\h])^*$. Set $\h^\lambda:=\{\xi+\langle\lambda,\xi\rangle,
\xi\in\h\}$. The algebra $(S(\g)/S(\g)\h^\lambda)^\h$ has the
natural Poisson structure induced from the Poisson structure on
$S(\g)$. Similarly, the space $(U(\g)/U(\g)\h^\lambda)^\h$ has a
natural structure of the associative algebra. The Duflo conjecture
states that for a certain element $\delta\in (\h/[\h,\h])^*$ the
filtered algebras
$\ZP\left((S(\g)/S(\g)\h^\lambda)^\h\right),\ZA\left((U(\g)/U(\g)\h^{\lambda-\delta})^\h\right)$
are isomorphic. Of course, the Duflo conjecture makes sense for
algebraic groups instead of Lie groups. Note, however, that
$(U(\g)/U(\g)\h^{\lambda-\delta})^\h$, in general, cannot be
considered as a quantization of $(S(\g)/S(\g)\h^\lambda)^\h)$
because the associated graded of the former does not necessarily
coincide with the latter.

One of the most significant partial results related to the Duflo
conjecture is Knop's construction of the Harish-Chandra isomorphism
for reductive group actions, \cite{Knop4}. Namely, let $X_0$ be a
smooth (not necessarily affine) variety and $G$ be a connected
reductive group (everything over $\K$). Set $X:=T^*X_0$. Let
$\D(X_0)$ denote the algebra of linear differential operators on
$X_0$. Knop constructed a certain monomorphism $\K[X]\cap
\ZP(\K(X)^G)\rightarrow \ZA(\D(X_0)^G)$, which becomes the usual
Harish-Chandra homomorphism when $X_0=G$. We remark that it is not
clear whether $\K[X]\cap \ZP(\K(X)^G)=\ZP(\K[X]^G)$ and whether the
image of the Harish-Chandra homomorphism coincides with whole
$\ZA(\D(X_0)^G)$. However, there is a  situation when both claims
hold:  $X_0$ is affine. We derive the existence of a weaker version
of the Harish-Chandra homomorphism in Section \ref{SECTION_Appl}.
This gives an alternative proof of the Duflo conjecture in the case
when both $G,H$ are reductive.

There are some other results related to the Duflo conjecture, see,
for example, \cite{Rybnikov},\cite{Torossian}.

Now let us describe the content of the paper. Sections
\ref{SECTION_Hamil},\ref{SECTION_quant} are preliminary. In the
former we review the definition of a Hamiltonian action in the
algebraic setting. Besides, we state there two results on
Hamiltonian actions obtained in \cite{alg_hamil},\cite{fibers} to be
used in the proof of Theorem \ref{Thm:1.1}. In Section
\ref{SECTION_quant} we review Fedosov's quantization in the context
of algebraic varieties and quantization of Hamiltonian actions.

Section \ref{SECTION_main} is the central part of this paper, where
Theorem \ref{Thm:1.1} is proved. At first, we prove Theorem
\ref{Thm:2.1}, which may be thought as a weaker version of Theorem
\ref{Thm:1.1}. In that theorem we establish an isomorphism with
desired properties between the algebras $A[[\hbar]],A_\hbar$, where
$A$ (resp., $A_\hbar$) is the integral closure of the image of
$S(\g)^\g$ in $\ZP(\K[X]^G)$ (resp., of the image of
$U^\wedge_\hbar(\g)^\g$ in $\ZA(\K[X][[\hbar]]^G)$). In its turn,
Theorem \ref{Thm:2.1} is derived from Proposition \ref{Prop:2.3},
which examines the integral closure of the image of
$U^\wedge_\hbar(\g)^\g$ in $\K[X^0][[\hbar]]^G$ for certain open
subvarieties $X^0\subset X$.

Finally, in Section \ref{SECTION_Appl} we discuss some applications
of Theorems \ref{Thm:1.1},\ref{Thm:2.1}. In particular, we give an
alternative proof of (a weaker form of) Knop's result cited above.
\section{Hamiltonian actions}\label{SECTION_Hamil}
In this section $G$ is a reductive algebraic group and $X$ is a
variety equipped with a regular symplectic form $\omega$ and an
action of $G$ by symplectomorphisms. Let $\{\cdot,\cdot\}$ denote
the Poisson bracket on $X$.

To any element $\xi\in\g$ one assigns the velocity vector field
$\xi_*$. Suppose  there is a linear map $\g\rightarrow \K[X],
\xi\mapsto H_\xi,$ satisfying the following two conditions:
\begin{itemize}
\item[(H1)] The map $\xi\mapsto H_\xi$ is $G$-equivariant.
\item[(H2)] $\{H_\xi,f\}=L_{\xi_*}f$ for any rational function $f$ on $X$ (here $L_{\bullet}$
denotes the Lie derivative).
\end{itemize}

\begin{defi}\label{Def:2.1.1}
The action  $G:X$ equipped with a linear map $\xi\mapsto H_\xi$
satisfying (H1),(H2) is said to be {\it Hamiltonian} and $X$ is
called a Hamiltonian $G$-variety. The functions $H_\xi$ are said to
be the {\it hamiltonians} of $X$.
\end{defi}

For a Hamiltonian action $G:X$ we define the morphism
$\mu:X\rightarrow \g^*$ (called the {\it moment map}) by the formula
\begin{equation*}%\label{eq:2.1:3}
\langle \mu(x),\xi\rangle= H_{\xi}(x),\xi\in\g,x\in X.
\end{equation*}

%Here and below we write $\xi_x$ instead of $\xi_{*x}$.

%Let us fix an invariant nondegenerate symmetric form on $\g$  and
%identify $\g$ and $\g^*$. In particular, we may consider $\mu$
%as a morphism from $X$ to $\g$.

Till the end of the section $X$ is affine.  Now we are going to
describe the algebra $\ZP(\K[X]^G)$. Let $\psi$ be the composition
of $\mu:X\rightarrow \g^*$ and the quotient morphism
$\g^*\rightarrow \g^*\quo G$. Denote by $A$ the integral closure of
$\psi^*(\K[\g^*]^\g)$ in $\K[X]^G$. Set $C_{G,X}:=\Spec(A)$ . We
have the decomposition $\psi=\tau\circ\widetilde{\psi}$, where
$\widetilde{\psi}$ is the dominant $G$-invariant morphism
$X\rightarrow C_{G,X}$ induced by the inclusion $A\hookrightarrow
\K[X]$ and $\tau$ is a finite morphism.

\begin{Prop}\label{Prop:1.11}
The morphism $\widetilde{\psi}$ is open. The equality
$\ZP(\K[X]^G)=\widetilde{\psi}^*(\K[\im\widetilde{\psi}])$ holds.
\end{Prop}
\begin{proof}
This follows from \cite{alg_hamil}, assertion 1 of Theorem 1.2.9 and
assertion 2 of Proposition 5.9.1.
\end{proof}

Now we are going to establish a certain decomposition result. Namely, we
locally decompose $X$ into the product of a vector space and of a
{\it coisotropic} Hamiltonian variety. Recall that a Hamiltonian
$G$-variety $X_0$ is said to be coisotropic if the Poisson field
$\ZP(\K(X)^G)$ is commutative or equivalently is algebraic over $A$.
Let $\pi$ denote the quotient morphism $X\rightarrow X\quo G$.

\begin{Prop}\label{Prop:1.12}
There is an open subset $Z^0\subset X\quo G$ such that
$\codim_{X\quo G}(X\quo G)\setminus Z^0\geqslant 2$ and for any
point $x\in X$ with $\pi(x)\in Z^0$ and closed $G$-orbit the
following condition holds:
\begin{itemize}
\item[($\clubsuit$)] There is a vector space $V$, a coisotropic
Hamiltonian $G$-variety $X_0$ and a point $x'\in
X_0\times\{0\}\subset X':=X_0\times V$ such that $X_0$ is a
homogeneous vector bundle over $Gx'$ and there is a Hamiltonian
isomorphism $\rho:X^\wedge_{Gx}\rightarrow X'^\wedge_{Gx'}$
("Hamiltonian" means that $\rho$ is a $G$-equivariant
symplectomorphism respecting the hamiltonians).
\end{itemize}
\end{Prop}
Here $X^\wedge_{Gx}, X'^\wedge_{Gx'}$ denote the completions of
$X,X'$ w.r.t. $Gx,Gx'$.
\begin{proof}
This follows from  \cite{fibers}, Corollary 3.10, and the symplectic
slice theorem, see \cite{slice} and \cite{Knop2}, Theorem 5.1.
\end{proof}

\section{Fedosov quantization}\label{SECTION_quant}
Let $B$ be a commutative associative algebra with unit equipped with
a Poisson bracket.

\begin{defi}\label{defi:1.1}
The map $*:B[[\hbar]]\otimes_{\K[[\hbar]]} B[[\hbar]]\rightarrow
B[[\hbar]]$ is called a {\it star-product} if it satisfies the
following conditions: \begin{itemize}
\item[(*1)] $*$ is $\K[[\hbar]]$-bilinear and continuous in the $\hbar$-adic
topology.
\item[(*2)] $*$ is associative, equivalently,  $(f*g)*h=f*(g*h)$ for all $f,g,h\in B$,
and $1\in B$ is a unit for $*$.
\item[(*3)] $f*g-fg\in \hbar B[[\hbar]],f*g-g*f-\hbar\{f,g\}\in \hbar^2B[[\hbar]]$ for all $f,g\in B$.
\end{itemize}
\end{defi}
By (*1), a star-product is uniquely determined by its restriction to
$B$. One may write $f*g=\sum_{i=0}^\infty D_i(f,g)\hbar^i,f,g\in B,
D_i:B\otimes B\rightarrow B$. Condition (*3) is equivalent to
$D_0(f,g)=fg, D_1(f,g)-D_1(g,f)=\{f,g\}$. If all $D_i$ are
bidifferential operators, then the star-product $*$ is called {\it
differential}. When we consider $B[[\hbar]]$ as an algebra w.r.t.
the star-product, we call it a {\it quantum algebra}.

\begin{Ex}\label{Ex:1.1}
Let $X=V$ be a finite-dimensional vector space equipped with a
constant nondegenerate Poisson bivector $P$.  The {\it Moyal-Weyl}
star-product on $\K[V][[\hbar]]$ (or $\K[[V,\hbar]]$) is defined by
$$f*g=\exp(\frac{\hbar}{2}P)f(x)\otimes g(y)|_{x=y}.$$
Here $P$ is considered as an element of $V\otimes V$. This space
acts naturally on $\K[V]\otimes \K[V]$.
\end{Ex}

Let  $G$ be an algebraic group acting on $B$ by automorphisms. It
makes sense to speak about  $G$-invariant star-products ($\hbar$ is
supposed to be $G$-invariant). Now let $\K^\times$ act on $B,
(t,a)\mapsto t.a$ by automorphisms. Consider the action
$\K^\times:B[[\hbar]]$ given by $t.\sum_{i=0}^\infty
a_j\hbar^j=\sum_{j=0}^\infty t^{jk}(t.a_j)\hbar^j$. If $\K^\times$
acts by automorphisms of $*$, then we say that $*$ is {\it homogeneous}.
Clearly, $*$ is homogeneous iff the map $D_l:B\otimes B\rightarrow
B$ is homogeneous of degree $-kl$.

 For
instance,  the Moyal-Weyl star-product $*$ is invariant with respect
to $\Sp(V)$. The action $\K^\times:V$ given by $(t,v)\mapsto
t^{-1}v$ makes $*$ homogeneous (for $k=2$).

Now we review the Fedosov approach (\cite{Fedosov1},\cite{Fedosov2})
to deformation quantization of smooth affine symplectic varieties.
Although Fedosov studied smooth real manifolds, his approach works
as well for smooth symplectic varieties and smooth formal schemes.
Let $X$ be a smooth variety with symplectic form $\omega$.

According to Fedosov, to construct a star-product one needs to fix a
symplectic connection on a variety in interest.

\begin{defi}\label{defi:1.3}
By a symplectic connection we mean  a torsion-free covariant
derivative $\nabla$ such that $\nabla\omega=0$.
\end{defi}

One can also define a symplectic connection on a smooth formal
scheme. In particular, if $\nabla$ is a symplectic connection on $X$
and $Y$ is a smooth subvariety of $X$, then $\nabla$ restricts to a
symplectic connection on the completion $X^\wedge_Y$.

It turns out that a symplectic connection on $X$ exists provided $X$
is affine. In fact, we need  a stronger version of this
claim.

\begin{Prop}[\cite{Wquant}, Proposition 2.22]\label{Prop:1.4}
Let $G$ be a reductive group  acting on $X$ by symplectomorphisms
and $\K^\times$ act on $X$ by $G$-equivariant automorphisms such
that $t.\omega= t^k\omega$ for some $k\in\Z$. Then there is a
$G\times\K^\times$-invariant symplectic connection $\nabla$ on $X$.
\end{Prop}

Fedosov constructed a differential star-product on $\K[X]$ starting
with a symplectic connection $\nabla$ and $\lambda\in
H^2_{DR}(X)[[\hbar]]$, see \cite{Fedosov2}, Section 5.3 or
\cite{GR_moment}. The element $\lambda$ is referred to as the {\it
characteristic cycle} of $*$. We remark that all intermediate
objects used in Fedosov's construction are obtained from some
regular objects (such as $\omega, \nabla$ or the curvature tensor of
$\nabla$) by a recursive procedure and so are regular too. If a
reductive group $G$ acts on $X$ by symplectomorphisms (resp.,
$\K^\times$ acts on $X$ such that $t.\omega=t^k\omega$), then $*$ is
$G$-invariant, resp., homogeneous.

In the sequel we will need the following result on equivalence of
two different Fedosov star-products on a formal scheme.

\begin{Prop}\label{Prop:1.5}
Let $G$ be a reductive group acting on affine symplectic varieties
$X,X'$ by symplectomorphisms, $\nabla,\nabla'$ be $G$-invariant
symplectic connections on $X,X'$, and $\lambda\in
H^2_{DR}(X)[[\hbar]],$ $ \lambda'\in H^2_{DR}(X')[[\hbar]]$.
Further, let $x,x'$ be points of $X,X'$ with closed $G$-orbits.
Suppose that there is a $G$-equivariant symplectomorphism
$\psi:X^\wedge_{Gx}\rightarrow X'^\wedge_{Gx'}$ such that
$\psi^*(\lambda'|_{Gx'})=\lambda|_{Gx}$. Then there are differential
operators $T_i,i=1,2,\ldots,$ on $X^\wedge_{Gx}$ such that
$(id+\sum_{i=1}^\infty T_i\hbar^i)\circ\psi^*$ is an isomorphism of
the  quantum algebras $\K[X']^{\wedge}_{Gx'}[[\hbar]],
\K[X]^\wedge_{Gx}[[\hbar]]$.
\end{Prop}

Again, to prove these we note that
$\psi^*(\lambda'|_{X^\wedge_{Gx'}})=\lambda|_{X^\wedge_{Gx}}$ and
repeat Fedosov's argument in  \cite{Fedosov2}, Theorem 5.5.3.

Now let us discuss quantization of Hamiltonian actions. Let $X$ be a
smooth affine symplectic variety and  $*$  a star-product on
$\K[X][[\hbar]]$. Let $G$ be a reductive group acting on $X$ by
automorphisms of $*$. We say that this action is {\it
$*$-Hamiltonian} if there is a $G$-equivariant
 linear map $\g\rightarrow \K[X][[\hbar]],\xi\mapsto \widehat{H}_\xi,$ satisfying the equality
\begin{equation}\label{eq:2.5}[\widehat{H}_\xi,f]=\hbar \xi_* f, \forall \xi\in\g, f\in \K[X].\end{equation}

The functions $\widehat{H}_\xi$ are said to be {\it quantum
hamiltonians} of the action.

Let $H_\xi$ be the classical part of $\widehat{H}_\xi$, that is,
$H_\xi\in \K[X], \widehat{H}_\xi\equiv H_\xi (\mod \hbar)$. Then the
map $\xi\mapsto H_\xi$ turns $X$ into a Hamiltonian $G$-variety.
Conversely, the following result takes place.

\begin{Thm}[\cite{GR_moment}, Theorem 6.2]\label{Thm:2.4}
Let $X$ be an affine symplectic Hamiltonian $G$-variety and $*$ be the star-product  on $\K[X][[\hbar]]$  obtained by the Fedosov
construction with a $G$-invariant connection $\nabla$ and
$\lambda\in H^2_{DR}(X)[[\hbar]]$. Then the following conditions are
equivalent:
\begin{enumerate} \item The $G$-variety $X$ has a $*$-Hamiltonian structure.
\item The 1-form $i_{\xi_*}\widetilde{\lambda}$ is exact for each $\xi\in\g$, where $\widetilde{\lambda}$
is a representative of $\lambda$ and $i_{\xi_*}\widetilde{\lambda}$
denotes the contraction of $\widetilde{\lambda}$ and $\xi_*$.
\end{enumerate}
\end{Thm}

Again, we remark that, although Gutt and Rawnsley dealt with smooth
manifolds, their results remain valid in the algebraic category as
well.

\begin{defi}\label{defi:1.13}
Let $\g$ be a Lie algebra.  By the {\it homogeneous enveloping
algebra} $U_\hbar(\g)$ we mean the quotient of $T(\g)[\hbar]$ by the
ideal generated by $\xi\eta-\eta\xi-\hbar[\xi,\eta], \xi,\eta\in\g$.
The {\it completed homogeneous enveloping algebra}
$U^\wedge_\hbar(\g)$ is, by definition,
$\varprojlim_{k\rightarrow\infty} U_\hbar(\g)/\hbar^k U_\hbar(\g)$.
\end{defi}

Below $\g$ is a reductive Lie algebra.  Elements of
$U^\wedge_\hbar(\g)$ are identified with formal power series
$\sum_{i=0}^\infty \xi_i\hbar^i,\xi_i\in\g$. It is clear that the
map $\xi\mapsto \widehat{H}_\xi$ is extended to a unique continuous
$\K[[\hbar]]$-algebra homomorphism $U^\wedge_\hbar(\g)\rightarrow
\K[X][[\hbar]]$. Note that the map $\sum_{i=0}^\infty
\xi_i\hbar^i\mapsto \xi_0$ is a $G$-equivariant algebra epimorphism
$U^\wedge_\hbar(\g)\twoheadrightarrow S(\g)$.

Now let us describe the center of $U^\wedge_\hbar(\g)$. Clearly,
$\ZA(U^\wedge_\hbar(\g))=U^\wedge_\hbar(\g)^\g$. Therefore there is
the natural epimorphism $\ZA(U^\wedge_\hbar(\g))\rightarrow
S(\g)^\g$. Since $S(\g)^\g$ is a polynomial algebra, there is a
section $S(\g)^\g\hookrightarrow \ZA(U^\wedge_\hbar(\g))$ of the
above epimorphism. Any such  section is extended to a topological
$\K[[\hbar]]$-algebra isomorphism $S(\g)^\g[[\hbar]]\rightarrow
\ZA(U^\wedge_\hbar(\g))$. Note that there are unique continuous
actions $\K^\times:S(\g)^\g[[\hbar]], \ZA(U^\wedge_\hbar(\g))$ such
that $t.\xi=t\xi, t.\hbar=t\hbar, \xi\in\g,t\in\K^\times$. An
isomorphism $S(\g)^\g[[\hbar]]\rightarrow \ZA(U^\wedge_\hbar(\g))$
can be chosen  $\K^\times$-equivariant. %Finally, let us note that
%the image of $U^\wedge_\hbar(\g)$ in $\K[X][[\hbar]]$ depends only
%on $X,G$ but not on the choice of $\widehat{H}_\xi$.

\section{The proof of the main theorem}\label{SECTION_main}
Set $A:=\K[C_{G,X}]$. Let $A_\hbar$ denote the integral closure of
the image of $\ZA(U^\wedge_\hbar(\g))$ in  $\ZA(\K[X][[\hbar]]^G)$.
We have the natural algebra homomorphism $A_\hbar\rightarrow A,
\sum_{i=0}^\infty f_i\hbar^i\mapsto f_0$.

Here is the main result of this section.
\begin{Thm}\label{Thm:2.1}
Suppose that $\K^\times$ acts on $X$ such that $*$ is homogeneous
with $t.\hbar=t^k\hbar, k\in\Z,
t.\widehat{H}_\xi=t^k\widehat{H}_\xi$. There is a
$\K^\times$-equivariant isomorphism $A[[\hbar]]\rightarrow A_\hbar$
of $\K[[\hbar]]$-algebras such that the following diagram is
commutative.

\begin{picture}(100,30)
\put(2,22){$S(\g)^\g[[\hbar]]$} \put(2,2){$\ZA(U^\wedge_\hbar(\g))$}
\put(35,22){$A[[\hbar]]$}\put(36,2){$A_\hbar$}\put(55,13){$A$}
\put(38,20){\vector(0,-1){13}}
\put(45,22.5){\vector(3,-2){10}}\put(41,5){\vector(3,2){13}}
\put(8,20){\vector(0,-1){13}}\put(10,13){{\tiny $\cong$}}
\put(20,23.5){\vector(1,0){14}} \put(20,4){\vector(1,0){14}}
\end{picture}
\end{Thm}

Let $B$ be a commutative algebra and $V$ a symplectic vector space.
We denote by $\hat{W}_V(B)$ the algebra $B[[V^*,\hbar]]$ equipped
with the Moyal-Weyl star-product.

\begin{Lem}\label{Lem:2.4}
Let $B$ be an integral domain and $V$ a symplectic vector space. If
$f\in \hat{W}_V(B)$ is integral over $B[[\hbar]]$, then $f\in
B[[\hbar]]$.
\end{Lem}
\begin{proof}
Let $F$ denote the algebraic closure of the fraction field
$\Quot(B)$ of $B$. There is the natural embedding
$\hat{W}_V(B)\hookrightarrow \hat{W}_V(F)$ and we may assume that
$B=F$. Let $P\in F[t]$ be a monic polynomial such that $P(f)=0$. The
algebraic closure of $\Quot(F[[\hbar]])$ coincide with the field
$F\{\hbar\}$ consisting of all expressions of the form $g(\hbar^{\alpha})$, where $\alpha\in \Q, g\in F[[\hbar]]$, see, for example,
\cite{Postnikov}. So there are $g_1,\ldots,g_k\in F\{\hbar\}$ such
that $P(t)=\prod_{i=1}^k (t-g_i)$. Let $\nu=\hbar^{1/r}$ be such
that $g_1,\ldots,g_k$ are Laurent power series in $\nu$. Consider the
algebra $\widetilde{W}$ that coincides with $F[[V^*,\nu]]$ as the
vector space and is equipped with the Moyal-Weyl star-product with
$\hbar=\nu^r$. The algebra $\hat{W}_V(F)$ is naturally embedded into
$\widetilde{W}$. For sufficiently large $n$ we have $\nu^n
P(f)=\prod_{i=1}^k P_i(f)$, where $P_i$ is a linear polynomial. To
prove the claim of the lemma it remains to verify that
$\widetilde{W}$ has no zero divisors. Indeed, let
$a=\sum_{i=0}^\infty a_i\nu^i, b=\sum_{i=0}^\infty b_i \nu^i,
a_i,b_i\in F[[V^*]]$ be such that $a*b=0$. Then $a_jb_l=0$, where
$j,l$ be the minimal integers such that $a_j\neq 0,b_l\neq 0$, which
is nonsense.
 \end{proof}

\begin{Prop}\label{Prop:2.3}
Denote by $\pi$ the quotient morphism $X\rightarrow X\quo G$. Let
$g$ lie in $\psi^*(S(\g)^\g)$. Set $X^0:=\{x\in X| g(x)\neq 0\}$.
 If an element $\widehat{f}\in \K[X^0][[\hbar]]^G$ is integral
over the image of $\ZA(U^\wedge_\hbar(\g))$, then
$\widehat{f}\in\ZA(\K[X][[\hbar]]^G)$.
\end{Prop}
\begin{proof}
It is enough to check that $\widehat{f}\in \ZA(\K[X^0][[\hbar]]^G)$
and that $\widehat{f}$ is defined in all points  $y\in(X\quo
G)^{reg}$ such that \begin{enumerate}\item $y$ satisfies condition
$(\clubsuit)$ of Proposition \ref{Prop:1.12},\item either $g(y)\neq
0$ or $y$ is a smooth point of the zero locus of $g$.\end{enumerate}
Let $X_0,V,X',x',\rho$ be such as in ($\clubsuit$). We may assume
that $\rho(x)=x'$ and identify $Gx\cong Gx'$.

Let $\lambda$ be the characteristic class used in the construction
of $*$. Choose a representative $\overline{\lambda}$ of $\lambda$ in
$\Omega^2(X)[[\hbar]]$. Set
$\overline{\lambda}'=\pi^*(\iota^*(\overline{\lambda}))$, where
$\iota$ denotes the inclusion $Gx\hookrightarrow X$ and $\pi$ the
projection $X'\twoheadrightarrow Gx$. Finally, let $\lambda'$ denote
the class of $\overline{\lambda}'$. It is easy to see that
$i_{\xi_*}\overline{\lambda}'$ is an exact form.  Construct the
star-product on $\K[X'][[\hbar]]$ w.r.t. some $G$-invariant
connection and the characteristic class $\lambda'$. By Theorem
\ref{Thm:2.4}, $X',X_0$ have $*$-Hamiltonian structures. Then $\rho$
induces an isomorphisms $\iota:\K[X]^\wedge_{Gx}\rightarrow
\K[X']^\wedge_{Gx'}, \K[X]^\wedge_{Gx}[g^{-1}]\rightarrow
\K[X']^\wedge_{Gx'}[\iota(g)^{-1}]$. The characteristic classes of
the star-products on $\K[X]_{Gx}^\wedge[[\hbar]],
\K[X']_{Gx}^\wedge[[\hbar]]$ coincide, so there is an isomorphism
$\iota_\hbar:=(id+\sum_{i=1}^\infty
T_i\hbar^i)\circ\iota:\K[X]_{Gx}^\wedge[[\hbar]]\rightarrow
\K[X']^\wedge_{Gx'}[[\hbar]]$ of quantum algebras, where all $T_i$
are differential operators, see Proposition \ref{Prop:1.5}. So
$\iota_\hbar$ is extended to the isomorphism $\iota_\hbar:
\K[X]^\wedge_{Gx}[g^{-1}][[\hbar]]\rightarrow
\K[X']^\wedge_{Gx'}[\iota(g)^{-1}][[\hbar]]$ of quantum algebras.
Taking $G$-invariants, we get the isomorphism $\K[X\quo
G]^\wedge_{y}[g^{-1}][[\hbar]]\cong \K[X'\quo
G]^\wedge_{y}[\iota(g)^{-1}][[\hbar]]$. From construction it is
clear that $\iota(g)\in \K[X_0]^G\hookrightarrow \K[X'\quo
G]^\wedge_{y}[\iota(g)^{-1}]=\K[X_0\quo
G]^\wedge_{y}[[V^{H}]][\iota(g)^{-1}]$.

Set $B:=\K[X_0\quo G]^\wedge_{y}, \widetilde{B}:=B[\iota(g)^{-1}]$.
Let us check that the structures of the classical and the quantum
algebras on $\widetilde{B}[[\hbar]]$ are isomorphic. The algebra $B$
is isomorphic to the formal power series algebra
$\K[[x_1,\ldots,x_n]]$. By the assumptions on $y$, we may assume
that either $g(y)\neq 0$, and then $\widetilde{B}=B$, or
$\widetilde{B}=\K[[x_1,\ldots,x_n]][x_1^{-1}]$. Now let us show that
the quantum algebra $B[[\hbar]]$ is commutative. Indeed, $\K[X_0]^G$
is algebraic over the image of $S(\g)^\g$, for $X_0$ is coisotropic.
From that and the observation that the star-product on
$\K[X_0]^G[[\hbar]]$ is differential we see that the quantum algebra
$\K[X_0\quo G][[\hbar]]$ is commutative. On the other hand,
$\K[X_0\quo G]$ is dense in $B$ and we are done. Since the
star-product on $B[[\hbar]]$ is differential, we see that the
quantum algebra $\widetilde{B}[[\hbar]]$ is also commutative. So
both classical and quantum algebras $B[[\hbar]]$ are commutative
complete local algebras. Therefore there are $q_1,\ldots,q_n\in
B[[\hbar]]$ such that the map $\varphi:x_i\rightarrow x_i+\hbar q_i$
defines the homomorphism from the classical algebra to the quantum
one. Since $\widetilde{B}[[\hbar]]$ is naturally identified with
$B[\varphi(x_1)^{-1}][[\hbar]]$, we see that $\varphi$ is extended
to the isomorphism $\widetilde{B}[[\hbar]]\rightarrow
\widetilde{B}[[\hbar]]$.

So we have the isomorphism of quantum algebras $$\K[X'\quo
G]^\wedge_{y}[\iota(g)^{-1}][[\hbar]]\cong
\widetilde{B}[[\hbar]]\widehat{\otimes}_{\K[[\hbar]]}\hat{W}_{V^H}(\K)\cong
\hat{W}_{V^H}(\widetilde{B}).$$ By Lemma \ref{Lem:2.4},
$\iota_\hbar(\widehat{f})\in \widetilde{B}[[\hbar]]$. Since
$\iota_\hbar(\widehat{f})$ is integral even over the image of
$\ZA(U^\wedge_\hbar(\g))$ in $B[[\hbar]]$, we easily get
$\iota_\hbar(\widehat{f})\in B[[\hbar]]$. It follows that
$\widehat{f}$ is defined in $y$ and commutes with the whole algebra
$\K[X\quo G]^\wedge_y$. This completes the proof.
\end{proof}

\begin{proof}[Proof of Theorem \ref{Thm:2.1}]
% Let $B$ be a subalgebra of $A$ containing $S(\g)^\g$ and $f\in A$.
%Define the subvariety $Y(B,f)\subset \Spec(B[f])$ as follows. We
%have the natural morphism $\tau_{B,f}:\Spec(B[f])\rightarrow
%\Spec(B)$. By definition, $Y(B,f)$ consists of all points $y$ of
%$\Spec(B[f])$ such that both $y, \tau_{B,f}(y)$ are smooth and the
%morphism $\tau_{B,f}$ is \'{e}tale in $y$.

%Let $f_1,\ldots,f_k$ be elements of $A$. Set
%$Y(f_1,\ldots,f_k):=Y(S(\k)[f_1,\ldots,f_{k-1}],f_k)$. In
%particular, the natural morphism $Y(f_1,\ldots,f_k)\rightarrow
%S(\g)^\g$ is \'{e}tale.

%Let us check that any point $x\in \im\widetilde{\psi}_{T,X}$ belongs
%to some subvariety of the  form $Y(f_1,\ldots,f_k)$, where
%$S(\g)^\g[f_1,\ldots,f_k]=A$. First of all, recall that the morphism
%$\tau:\im\widetilde{\psi}\rightarrow \t$ is \'{e}tale. There
%are $f\in A, g\in S(\g)^\g$ such that $g(\tau(x))\neq 0, f(x)=0,
%A[g^{-1}]=S(\g)^\g[f,g^{-1}]$. Set $f_1=f$ and choose arbitrary
%$f_2,\ldots,f_k$ such that $A=S(\g)^\g[f_1,\ldots,f_k]$. The inclusion
%$x\in Y(f_1,\ldots,f_k)$ follows from
%$A[g^{-1]=S(\g)^\g[f_1,\ldots,f_{k-1}][g^{-1}]=\ldots=S(\g)^\g[f_1][g^{-1}]$
%and the fact that $\tau:\Spec(A)\rightarrow\t$ is \'{e}tale in $x$.

By the graded version of the Noether normalization theorem, there are algebraically
independent homogeneous elements $g_1,\ldots,g_k\in \mu^*(S(\g)^\g)$ that are algebraically
independent such that $\mu^*(S(\g)^\g)$ is  finite over $B_0:=\K[g_1,\ldots, g_k]$.

Let $B$ be a subalgebra of $A$ containing $B_0$. Suppose there
is a continuous homomorphism $\iota:B[[\hbar]]\rightarrow A_\hbar$
of $\K[[\hbar]]$-algebras such that the diagram analogous to that of
the theorem (with $B$ instead of $A$) is commutative. Automatically,
the image of $\iota$ is closed and $\iota$ is a topological
isomorphism onto its image.

Clearly, $B=B_0$ satisfies the above conditions.
 By definition, $A$ is a finite $B_0$-module, so
 any ascending chain of subalgebras between
$B_0$ and $A$ is finite. So we assume that $B$ is a maximal
subalgebra of $A$ satisfying the conditions of the previous
paragraph.  Theorem \ref{Thm:2.1} is equivalent to the equality
$A=B$. Assume that $A\neq B$.

Choose some homogeneous elements $a_0,\ldots,a_{n-1}\in B,a_0\neq 0$
and $f\in A$. Set $P(t)=\sum_{i=0}^n a_it^i, a_n:=1$. Suppose $P(f)=0$ and
$Q(f)\neq 0$ for any monic polynomial $Q\in B[x]$ with $\deg Q<\deg
P$.
 Set $\widehat{a}_{i}=\iota(a_i), i=\overline{0,n-1}$. We are going
to show that there exists a  homogeneous element
$\widehat{f}=\sum_{i=0}^\infty f_i\hbar^i\in \K[X][[\hbar]]^G$ such
that $f_0=f$ and
\begin{itemize}
\item[(*)]$\widehat{f}^{*n}+\widehat{a}_{n-1}*\widehat{f}^{*(n-1)}+\ldots+\widehat{a}_0=0$.
\end{itemize}

Suppose  we have already constructed such $\widehat{f}$. Let
$\widetilde{B}$ denote the subalgebra in $A$ generated by $B$ and
$f$. By Proposition \ref{Prop:2.3}, $\widehat{f}$ is a central
element of $\K[X][[\hbar]]^G$. Let $\widetilde{\iota}$ denote the
continuous $\K[[\hbar]]$-algebra homomorphism
$\widetilde{B}[[\hbar]]\rightarrow \K[X][[\hbar]]^G$ defined by
$\widetilde{\iota}(b)=\iota(b),b\in
B,\widetilde{\iota}(f)=\widehat{f}$. By construction,
$\widetilde{\iota}$ is well-defined and satisfies the assumptions of
the first paragraph of the proof.

At first, we show that there is an open $G$-stable subvariety $X^1$
such that there is an element $\widehat{f}\in \K[X^1][[\hbar]]^G$
satisfying (*). Namely, for $X^1$ we take the set of all points
$x\in X$ such that $\frac{dP}{dt}(f)$ is nonzero in $x$. Since $X^1$
contains $X^0$ of the form indicated in Proposition \ref{Prop:2.3},
we will automatically  get $\widehat{f}\in \K[X][[\hbar]]^G$.

We will construct such $\widehat{f}$ recursively. Clearly, if
$\widehat{f}$ satisfies (*) iff $\widehat{f}_{(m)}:=\sum_{i=0}^m
f_i\hbar^i$ satisfies
\begin{itemize}
\item[$(*_m)$]
$\widehat{f}_{(m)}^{*n}+\widehat{a}_{n-1}*\widehat{f}_{(m)}^{*(n-1)}+\ldots+\widehat{a}_0\in\hbar^{m+1}\K[X][[\hbar]]$
\end{itemize}
for any $m\in \N$.

Suppose we have already found $f_1,\ldots,f_m\in \K[X^1]^G$ such
that $\widehat{f}_{(m)}$ satisfies $(*_m)$. Let us check that  there
is a unique element $\widehat{f}_{(m+1)}\in \K[X^1]^G$ such that
$\widehat{f}_{m+1}$ satisfies $(*_{m+1})$. This follows from the
observation that the coefficient of $\hbar^{m+1}$ in the l.h.s. of
$(*_{m+1})$ is equal to $\frac{dP}{dt}(f)f_{m+1}-Q$, where $Q$
depends only on $f_0,\ldots,f_m$. So $f_{m+1}$ is constructed. By
construction, $f_{m+1}$ is homogeneous, the degree of
$f_{m+1}\hbar^{m+1}$ coincides with that of $f$ and
$\widehat{f}_{(m+1)}$ satisfies $(*_{m+1})$.
\end{proof}

\begin{proof}[Proof of Theorem \ref{Thm:1.1}]
Let us construct an  algebra homomorphism $\ZP(\K[X]^G)\rightarrow
\ZA(\K[X][[\hbar]]^G)$ that is a section of
$\ZA(\K[X][[\hbar]]^G)\rightarrow \ZP(\K[X]^G), \sum_{i=0}^\infty
f_i\hbar^i\mapsto f_0$. Let $\widehat{\iota}$ denote an isomorphism
$A[[\hbar]]\rightarrow A_\hbar$ constructed in Theorem
\ref{Thm:2.1}.

By Proposition \ref{Prop:1.11},
$\ZP(\K[X]^G)=\widetilde{\psi}^*(\K[\im\widetilde{\psi}])$. So for
any $f\in \ZP(\K[X]^G)$ there exist elements $f_1,\ldots,f_k,
g_1,\ldots,g_k\in A$ such that $fg_i=f_i$ and for any $y\in
\im\widetilde{\psi}$ there is $i$ with $g_i(y)\neq 0$. Then
$\frac{\widehat{\iota}(f_i)}{\widehat{\iota}(g_i)}=
\frac{\widehat{\iota}(f_j)}{\widehat{\iota}(g_j)}$. By construction
$\iota(g_i)-g_i\in \hbar A_\hbar$, so if $g_i(y)\neq 0$, then
$\frac{\widehat{\iota}(f_i)}{\widehat{\iota}(g_i)}$ is defined in
$y$. So the fractions
$\frac{\widehat{\iota}(f_i)}{\widehat{\iota}(g_i)}$ are glued
together into an element $\widehat{f}\in \K[X][[\hbar]]^G$. Since
$\widehat{\iota}(f_i),\widehat{\iota}(g_i)\in
\ZA(\K[X][[\hbar]]^G)$, we get $\widehat{f}\in
\ZA(\K[X][[\hbar]]^G)$. Set $\widehat{\iota}(f):=\widehat{f}$. It is
clear from the construction that the map
$\widehat{\iota}:\ZP(\K[X]^G)\rightarrow \ZA(\K[X][[\hbar]]^G)$ has
the desired properties.

Let us lift $\widehat{\iota}:\ZP(\K[X]^G)\rightarrow
\ZA(\K[X][[\hbar]]^G)$ to the continuous $\K[[\hbar]]$-algebra
homomorphism $\iota_\hbar:\ZP(\K[X]^G)[[\hbar]]\rightarrow
\ZA(\K[X][[\hbar]]^G)$. By the construction of $\iota$, the
homomorphism $\iota_\hbar$ is injective and makes the right triangle
of the theorem diagram  commutative. From the commutativity of the
triangle one easily deduces that this homomorphism is also
surjective. It follows from the properties of the isomorphism
$A[[\hbar]]\rightarrow A_\hbar$ that the left square is also
commutative.
\end{proof}

\section{Some special cases}\label{SECTION_Appl}
At first, let $X=V$ be a symplectic vector space with  constant
symplectic form and Moyal-Weyl star-product $*$. Suppose $G$ is a
connected reductive group acting on $V$ by linear
symplectomorphisms. This action is Hamiltonian with quadratic
hamiltonians.  Equip $V$ with the action of $\K^\times$ given by
$(t,v)\mapsto t^{-1}v$. Setting $t.\hbar=t^2\hbar$ we make $*$ a
homogeneous star-product. The quantum hamiltonians are homogeneous
of degree 2. By \cite{alg_hamil}, Theorem 1.2.7,
$\widetilde{\psi}_{G,V}$ is surjective whence
$\ZP(\K[V]^G)=\K[C_{G,V}]$. Further, (\cite{fibers}, Corollary 3.12.
or \cite{Knop6}, Section 1) $C_{G,V}$ is an affine space. By Theorem
\ref{Thm:2.1}, there is a $\K^\times$-equivariant isomorphism
$\K[C_{G,V}][[\hbar]]\rightarrow \ZA(\K[V][[\hbar]]^G)$. The
$\K^\times$-finite parts of these two algebras coincide with
$\K[C_{G,V}][\hbar],\ZA(\K[V][\hbar]^G)$. Taking quotients of these
algebras by the ideal generated by $\hbar-1$, we get an isomorphism
$\K[C_{G,V}]\rightarrow \ZA(W(V)^G)$, where $W(V)$ is the Weyl
algebra of $V$. In particular, $\ZA(W(V)^G)$ is a polynomial
algebra.

Now we consider the case when $X=T^*X_0$ for some smooth affine
$G$-variety $X_0$. We consider the action $\K^\times:X$ given by
$t.(x_0,\alpha)=(x_0,t^{-1}\alpha), t\in\K^\times, x_0\in
X_0,\alpha\in T^*_{x_0}X_0$. Choose a $G\times \K^\times$-invariant
symplectic connection $\nabla$ on $X$. Construct the star-product on
$X$ by means of $\nabla$ and the zero characteristic class. Setting
$t.\hbar=t\hbar$, we make $*$ a homogeneous star-product. Again, the
morphism $\widetilde{\psi}$ is surjective and $C_{G,X}$ is a
polynomial algebra. The latter follows from  results of
\cite{Knop1}. The quantum algebra $\K[X][[\hbar]]$ is
$G\times\K^\times$-equivariantly isomorphic to the {\it completed
homogeneous} algebra  $\D^\wedge_\hbar(X_0)$ of differential
operators on $X_0$, which is defined as follows.  Let
$\operatorname{Der}(X_0)$ denotes the space of all vector fields on
$X_0$. By the {\it homogeneous} algebra of differential operators on
$X_0$ we mean the quotient $\D_\hbar(X_0)$ of $T(\K[X_0]\oplus
\operatorname{Der}(X_0))[\hbar]$ by the relations $f\otimes g=fg,
\xi\otimes f-f\otimes \xi=\hbar \xi.f,
\xi\otimes\eta-\eta\otimes\xi=\hbar [\xi,\eta], f,g\in \K[X_0],
\xi,\eta\in \operatorname{Der}(X_0)$. We have a natural action
$G\times\K^\times:\D_\hbar(X_0)$, the group $\K^\times$ acts as
follows: $t.f=f, t.\xi=t\xi, t.\hbar=\hbar, f\in \K[X_0],\xi\in
\operatorname{Der}(X_0)$. By definition, $\D_\hbar^\wedge(X_0)$ is
the completion of $\D_\hbar(X_0)$ in the $\hbar$-adic topology.

For a particular choice of $\nabla$ an isomorphism
$\K[X][[\hbar]]\rightarrow \D_\hbar^\wedge(X_0)$ was constructed in
\cite{BNW}\footnote{ADDED IN PROOF. In fact, the quantization $\D^\wedge_\hbar(X_0)$ of $\K[X]$
corresponds not to the zero characteristic class but to the half of the Chern
class of the canonical bundle. This is implicitly contained in [BNW].
However, the precise value of the correspondent
characteristic class is not important for our argument. Indeed, any quantization
of $\K[X]$ is isomorphic to a Fedosov one with {\it some} characteristic class.
Theorem 1.1 works for an arbitrary characteristic class as long as the comoment
map can be quantized. The last condition definitely holds for $\D_\hbar^\wedge(X_0)$.} and the algebra $\K[X][[\hbar]]$ does not depend (up to a
$G\times\K^\times$-equivariant isomorphism) on the choice of
$\nabla$. So we get an isomorphism $\K[C_{G,X}]\cong
\ZA(\D(X_0)^G)$. This is a weak version of the main theorem of
\cite{Knop4}.

\bigskip
{\Small Department of Mathematics, Massachusetts Institute of
Technology, 77 Massachusetts Avenue, Cambridge, MA 02139, USA.

\noindent E-mail address: ivanlosev@math.mit.edu}
\end{document}